\documentclass[12pt]{amsart}  \usepackage{amsmath,amssymb,enumerate,amscd,tikz,url}

\makeatletter
\@namedef{subjclassname@2020}{%
\textup{2020} Mathematics Subject Classification}
\makeatother

\numberwithin{equation}{section}

\newtheorem{theorem}[equation]{Theorem}
\newtheorem{lemma}[equation]{Lemma}
\newtheorem{cor}[equation]{Corollary}
\newtheorem{prop}[equation]{Proposition}

\theoremstyle{definition}

\newtheorem{Ex}[equation]{Example}
\newtheorem{Rem}[equation]{Remark} 
\newtheorem{Not}[equation]{Notation}

\newcommand{\C}{{\mathbb C}}
\newcommand{\F}{{\mathbb F}}

\newcommand{\N}{{\mathbb N}}
\newcommand{\Q}{{\mathbb Q}}
\newcommand{\R}{{\mathbb R}}

\newcommand{\Z}{{\mathbb Z}}

\newcommand{\Mu}{\boldsymbol \mu}

\newcommand{\gp}{{\mathfrak p}}

\newcommand{\cD}{{\mathcal D}}
\newcommand{\cG}{{\mathcal G}}
\newcommand{\cO}{{\mathcal O}}

\newcommand{\Gal}{\operatorname{Gal}}
\newcommand{\GL}{\operatorname{GL}}

\newcommand{\mult}{\operatorname{mult}}
\newcommand{\Tr}{\operatorname{Tr}}

\newcommand{\ord}{\operatorname{ord}}

\newcommand{\Frob}{{\rm Frob}}

\newcommand{\lr}[1]{{\langle {#1} \rangle}}
\newcommand{\vv}[1]{{\left\vert {#1} \right\vert}}
\newcommand{\fpart}[1]{{\left\{ {#1}  \right\}}}
\newcommand{\ov}[1]{\overline{#1}}
\newcommand{\fdeg}[2]{[{#1}\!:\!{#2}]}

\newcommand{\floor}[1]{ {\left\lfloor {#1} \right\rfloor} }

\makeatletter
\renewcommand*\l@subsection{\@tocline{2}{0pt}{30pt}{0pt}{}}
\makeatother

\title{Distinguishing Gauss sums}

\author[M. Adrian]{Moshe Adrian}
\address{Department of Mathematics \\
Queens College, CUNY \\
65-30 Kissena Blvd., Queens, NY 11367-15971}
\email{moshe.adrian@qc.cuny.edu}

\author[J. Diamond]{Jack Diamond}
\address{Department of Mathematics \\
Queens College, CUNY \\
65-30 Kissena Blvd., Queens, NY 11367-15971}
\email{jack.diamond@qc.cuny.edu}

\author[K. Kramer]{Kenneth Kramer}
\address{Department of Mathematics \\
Queens College, CUNY \\
65-30 Kissena Blvd., Queens, NY 11367-15971}
\email{kkramer@qc.cuny.edu}

\author[G. K.-F. Tam]{Geo Kam-Fai Tam}
\address{Mathematics Department  \\ Xiamen University Malaysia \\ Jalan Sunsuria, Bandar Sunsuria, 43900
Sepang, Selangor, Malaysia.}
\email{kamfai.tam@xmu.edu.my}

\date{22 Sept 2025:  revision started.}

\begin{document}

\baselineskip=17pt

\begin{abstract}
Let $\F_q$ be the field of order $q = p^f$ for prime $p$. If $G(\chi)$ is the Gauss sum attached to a multiplicative character $\chi$ on $\F_q^\times$, then $G(\chi) = G(\chi^p)$.  We investigate the converse question:  when does the equality of Gauss sums $G(\chi_2) = G(\chi_1)$ imply that $\chi_2 = \chi_1^{p^j}$ for some integer $j$.  If $\chi_2 = \chi_1^{p^j}$, we say that $\chi_1$ and $\chi_2$ are Frobenius-conjugate.

We use the Stickelberger factorization of ideals in cyclotomic fields to give an easily testable criterion for equality of Gauss sums, based on $p$-adic digit expansion.  As an application, we develop several conditions under which Gauss sum equalities between characters on $\F_q^\times$ are explained by Frobenius-conjugacy.  For example, if $\chi_1$ has order $q-1$ or $\frac{1}{2}(q-1)$ and $G(\chi_2) = G(\chi_1)$, then $\chi_1$ and $\chi_2$ are Frobenius-conjugate.

We also include several examples, based on the explicit evaluation of certain {\em pure} Gauss sums by R.J. Evans.
\end{abstract}

\maketitle
\tableofcontents



\section{Introduction}

Throughout this paper, $p$ is a prime, $q = p^f$ and $\mathbb{F}_q$ is the field with $q$ elements.   Let $\Mu_n$ denote the group of $n$-th roots of unity and choose a generator $\zeta_p$ for $\Mu_p$. Fix the additive character $\psi\!: \F_q \to \Mu_p$, defined by $\psi(a) = \zeta_p^{\Tr(a)}$, where $\Tr\!: \, \F_q \to \F_p$ is the trace.  The {\em Gauss sum} of a multiplicative character  $\chi\!: \, \F_q^\times \to \Mu_{q-1}$ is defined as
$$
G(\chi) = \sum_{a \in \F_q^\times} \, \chi(a) \, \psi(a) \quad \text{in} \quad \Z[\Mu_{q-1},\Mu_p],
$$
suppressing the dependence on the choice of additive character.  With this normalization, we have $G(\bf 1) = -1$ for the trivial character $\chi = {\bf 1}$.  

The elements $\sigma$ of $\Gal(\F_q /\F_p)$ are generated by the Frobenius $x \mapsto x^p$ and they act on $\chi$ by $\chi^{\sigma}(a) = \chi(\sigma(a))$.  We say that $\chi$ and $\chi^{\sigma}$ are {\em Frobenius-conjugate}.  It is easy to see that $G(\chi^\sigma) = G(\chi)$.   We investigate the converse question:  when does the equality of Gauss sums $G(\chi_2) = G(\chi_1)$ imply that $\chi_1$ and $\chi_2$ are Frobenius conjugate?  By a {\em non-trivial} equality of Gauss sums, we mean one for which the characters are not Frobenius-conjugate.

The Gauss sum $G(\chi)$ is an invariant of the character $\chi$.  Our converse question was motivated by similar and related questions in the representation theory of $p$-adic groups, also converse questions.  Attached to a representation $\pi$ of $\GL(n,F)$, where $F$ is a local non-Archimedean field, are certain invariants called \emph{gamma factors}.  One may ask whether these invariants determine $\pi$.  Indeed, Gauss sums are a key ingredient in the definition of the gamma factor, so the results we prove here  have implications for the representation theory of $\GL(n,F)$.

In \S\ref{Criterion}, we recall basic Galois-theoretic and $p$-adic information about Gauss sums.  The result in Proposition \ref{EqualG} gives an easily testable criterion for Gauss sum equality.  In \S\ref{DistinctionResults}, we prove several criteria under which only trivial Gauss sum equalities hold.   For a special case treated in Proposition \ref{Twist}, we show that if the equality $G(\chi \otimes \widehat{\eta}) = G(\chi' \otimes \widehat{\eta})$ holds for all the twists by characters $\eta$ on $\F_p^\times$, then $\chi$ and $\chi'$ are Frobenius-conjugate (see \cite{NZY} for a thorough answer to this question).

In \S \ref{HDGK}, we recall two important formulas for the evaluation of Gauss sums.  The Hasse-Davenport formula applies when a character on $\F_q$ is lifted to a larger finite field.  The Gross-Koblitz formula provides an evaluation of $G(\chi)$ in terms of $p$-adic Gamma functions.  We use the explicit evaluation of certain {\em pure} Gauss sums \cite[Cor. 8]{Ev} to create examples in \S \ref{Explicit}.

\section{Criterion for equality of Gauss sums} \label{Criterion}

This section contains our necessary and sufficient criterion for equality of Gauss sums, proved in Proposition \ref{EqualG} by Galois theory.  For the reader's convenience, we include several basic properties of Gauss sums.   See \cite[\S6.1--6.2, 87--100]{W}, \cite[IV, \S3, 90--97]{L} and \cite[Ch.\! 11]{BEW} for details.  We first establish our Galois theory notation.

Fix a prime $p$ and let $q = p^f$.  Let $F = \Q(\Mu_p)$, $K = \Q(\Mu_{q-1})$ and $L = KF = \Q(\Mu_{p(q-1)})$.  Denote the ring of integers of $L$ by  $\cO_L$, choose a prime $\gp$ over $p$ in $L$ and identify the residue field $\cO_L/\gp$ with $\F_q$.  Let $\omega\!: \, \F_q^\times \xrightarrow{\sim} \Mu_{q-1}$ be the {\em Teichm\"uller character} at $\gp$, such that $\omega(a)$ is the unique root of unity in $\Mu_{q-1}$ congruent to $a$ modulo $\gp$.  Then $\omega$ generates the group of multiplicative characters of $\F_q^\times$.  

We have the following diagram of fields, with additional notation defined subsequently:  

\vspace{2  pt}

\begin{center}
\begin{tikzpicture}[scale=1, thin, baseline=(current  bounding  box.center)]	
       \draw (0,0) node {$\bullet$};  \draw (1,.4) node {$\bullet$}; \draw (.-.4,1) node {$\bullet$}; \draw (.6,1.4) node {$\bullet$};   \draw (.2,-.5) node {$\bullet$};  \draw (1.2,-.1) node {$\bullet$}; 
       \draw (0,0) -- (1,.4) -- (.6,1.4) -- (-.4,1) -- (0,0);   \draw (1,.4) -- (1.2,-.1) -- (.2,-.5)--(0,0); 
       \draw (-0.4,0) node {$K_0$};   \draw (-1.75,1.05) node {$K = \Q(\Mu_{q-1})$};  
       \draw (-.1,-.6) node {$\Q$};   \draw (2.3,.-.2) node {$F = \Q(\Mu_p)$}; 
       \draw (1.8,.4) node {$K_0(\Mu_p)$};  \draw (2.7,1.45) node {$L = KF = \Q(\Mu_{p(q-1)})$};
       \draw (1.2,.95) node {$\lr{\Phi}$};  \draw (0,1.55) node {$\lr{\gamma}$};
\end{tikzpicture}
$\begin{array}{l} 
\\
\Phi = \Frob_\gp, \hspace{3 pt} \Phi(\zeta_{q-1}) = \zeta_{q-1}^p,  \vspace{5 pt} \\ \gamma(\zeta_p) = \zeta_p^\alpha,   \hspace{3 pt} \lr{\alpha} = (\Z/p\Z)^\times.  \end{array}$
\end{center}
The group $\Gal(F/\Q)$ is cyclic of order $p-1$ and $\Gal(K/\Q) \simeq (\Z/(q-1)\Z)^\times$ is abelian of order $\phi(q-1)$, where $\phi$ is Euler's totient-function.   Moreover 
$$
\Gal(L/\Q) \simeq \Gal(K/\Q) \times \Gal(F/\Q).
$$  
Since the extension $L/\Q$ is abelian, the decompositon group $\cD_\gp(L/\Q)$ at $\gp$ does not depend on the choice of $\gp$ over $p$.  Let $\Phi$ be a Frobenius in $\cD_\gp(L/F)$, so that it acts trivially on $\Mu_p$.  Then $\Phi$ has order $f$ and its action on $K$ is determined by $\Phi(\zeta_{q-1}) = \zeta_{q-1}^p$.  If $K_0$ denotes the subfield of $K$ fixed by $\Phi$, then $\fdeg{K_0}{\Q} = \phi(q-1)/f$.  

Let the integer $\alpha$ be a primitive root mod $p$, i.e., $\alpha \bmod{p}$ generates $\F_p^\times$, and let $\gamma$ be the generator of $\Gal(L/K)$ characterized by $\gamma(\zeta_p) = \zeta_p^\alpha$.   Then the decomposition group is given by $\cD_\gp(L/\Q) = \lr{\Phi,\gamma}$ and $\Gal(L/K) = \lr{\gamma}$ is its inertia subgroup.  The action of $\cD_\gp(L/\Q)$ on Gauss sums is determined by 
\begin{equation}  \label{DpOnGaussSum}
\Phi(G(\chi)) = G(\chi) \quad \text{and} \quad \gamma(G(\chi)) =  \chi(\alpha^{-1}) \, G(\chi).
\end{equation}
In addition, Gauss sums admit an action of 
$$
\Gal(L/F) \simeq \Gal(K/\Q) \simeq (\Z/(q-1)\Z)^\times.
$$
Every element of this Galois group has the form $\sigma_u$, where $\sigma_u(\zeta_p) = \zeta_p$ and $\sigma_u(\zeta_{q-1}) = \zeta_{q-1}^u$, with $u \in (\Z/(q-1)\Z)^\times$.  Then we have $\sigma_u(G(\chi)) = G(\chi^u)$.

\begin{Not} \label{DigitSum}
Let $q = p^f$ and let $\ov{a}$ be a congruence class in $\Z/(q-1)\Z$.  Let $a$ be the representative for $\ov{a}$ satisfying $0 \le a \le q-2$ and write the $p$-adic expansion
\begin{equation} \label{pAdicDigits}
a = a_0 + a_1p + a_2 p^2 + \dots + a_{f-1}p^{f-1}  \,\text{ with } \, 0 \le a_j \le p-1 \text{ for all } j.
\end{equation}
The coefficients $a_j$ will be called the {\em $p$-adic digits} of $\ov{a}$ and $s_q(\ov{a}) = \sum_{j=0}^{f-1} a_j$ will be called the {\em $p$-adic digit sum}.  Define 
\begin{equation} \label{NuDef}
\begin{array}{l l l}
\nu_q(\ov{a}) = (a_0!) \, (a_1!) \cdots (a_{f-1}!) &\quad \text{if } p \text{ is odd and}  \vspace{2 pt}\\
\beta_q(\ov{a}) =  a_0 a_1 + a_1 a_2 + \dots + a_{f-1} a_0 &\quad \text{if } p=2.
\end{array}
\end{equation}
 If $c$ is in $\Z_p$, we write $s_q(c)$ for $s_q(\ov{c})$, $\nu_q(c)$ for $\nu_q(\ov{c})$ and $\beta_q(c)$ for $\beta_q(\ov{c})$ when there is no risk of ambiguity. 
\end{Not}

\vspace{5 pt}

Next, we collect some basic properties of the $p$-adic digits.  If $x$ is in $\R$, its floor $\floor{x}$ is the greatest integer less than or equal to $x$ and $\fpart{x} = x-\floor{x}$ denotes the fractional part of $x$.  Thus, the coset of $x$ in $\R/\Z$ is represented by $\fpart{x}$, with $0 \le \fpart{x} < 1$. 

\begin{lemma}  \label{cs}
Let $a = a_0 + a_1p + \dots + a_{f-1}p^{f-1}$ with $0 \le a_j \le p-1$, not all equal to $p-1$.
\begin{enumerate}[\hspace{2 pt}\rm i)]
\item The Frobenius automorphism $\Phi = \Frob_\gp$ acts on $\chi = \omega^a$ by $\Phi(\chi) = \omega^{\tilde{a}}$, where 
$
\tilde{a} = a_{f-1}+a_0p + a_1p^2 + \dots + a_{f-2}p^{f-1}.
$

\vspace{4 pt}

\item We have $s_q(-a) + s_q(a) = (p-1)f$ \, and \, $s_q(a) \equiv a \bmod{(p-1)}$.  \vspace{4 pt}
\item If $\fpart{x} = \text{fractional part of } x$, then $\displaystyle{\frac{s_q(a)}{p-1} =  \sum_{i=0}^{f-1} \left\{\frac{p^i a}{p^f-1}\right\}}$.
\end{enumerate}
\end{lemma}

See \cite[IV, \S3, Thm.\! 9]{L} for the following Stickelberger factorization of the ideal in $\cO_L$ generated by $G(\chi)$ and a congruence on the unit part of $G(\chi)$.  Recall that we chose a prime ideal $\gp$ over $p$ in $\cO_L$ to define the Teichm\"{u}ller character $\omega$.  There is a unique prime ideal over $p$ in $\cO_F$, generated for example by $\zeta_p - 1$.  Since $\Gal(L/F)$ acts transitively on the prime ideals of $\cO_L$ over the $\cO_F$-ideal $(\zeta_p-1)$, each has the form $\sigma^{-1}_u(\gp)$ for some $u$ in $(\Z/(q-1)\Z)^\times$.  

\begin{lemma}  \label{Stick}
We have $\ord_{\sigma^{-1}_u(\gp)} (G(\omega^{-c})) = s_q(uc)$.  If $p$ is odd, then
$$
\frac{G(\omega^{-c})}{(\zeta_p-1)^{s_q(c)}} \equiv - \frac{1}{\nu_q(c)} \pmod{\gp}.  
$$
\end{lemma}

\vspace{3 pt}

\begin{Rem}  \label{Nu2}
When $\gp \vert 2$ and $f \ge 3$, we need a stronger congruence:
$$
\frac{G(\omega^{-c})}{(-2)^{s_q(c)}} \equiv (-1)^{\beta_q(c)+1} \pmod{\gp^2},  \vspace{2 pt}
$$
for $s_q$ and $\beta_q$ defined in \eqref{NuDef}.  See Proposition \ref{2AdicCong}, where it is verified with the aid of the Gross-Koblitz formula.  
\end{Rem}

\vspace{5 pt}

It is well-known that if $\chi$ is not the trivial character, then 
\begin{equation}  \label{ConjChi}
G(\chi) G(\chi^{-1}) = \chi(-1) q
\, \text{ and } \,  \vert G(\chi) \vert_v = \sqrt{q} \,  \text{ for all } v \vert \infty.
\end{equation}

\vspace{5 pt}

The following Proposition is useful for distinguishing Gauss sums.  Let $U_{q-1} = (\Z/(q-1)\Z)^\times$.

\begin{prop}  \label{EqualG}
Let $\chi = \omega^{-c}$ and $\chi' = \omega^{-c'}$.  If $p = 2$, assume $f \ge 3$.  The equality of Gauss sums $G(\chi) = G(\chi')$ holds if and only if both of the following conditions are true:  \vspace{4 pt}  

{\rm i)} \, $s_q(uc) = s_q(uc')$ for all $u \in U_{q-1}$ \vspace{5 pt} and \hspace{10 pt} 

{\rm ii)} \, $\begin{cases} \nu_q(c)  \equiv \nu_q(c') \pmod{p} &\text{if } p \text{ is odd}, \\[2 pt]
           \beta_q(c)  \equiv \beta_q(c')  \pmod{2} &\text{if } p = 2 \text{ and } f \ge 3. \end{cases}$ 
\end{prop}

\begin{proof}
If $G(\chi) = G(\chi')$, then (i) and (ii) follow from Lemma \ref{Stick} and Remark \ref{Nu2}.  For the converse, let $\theta = G(\chi')/G(\chi)$ and assume that $s_q(c) = s_q(c')$.  Then Lemma \ref{cs}(ii) implies that $c \equiv c' \pmod{p-1}$.  It follows that if $\alpha$ is a primitive root modulo $p$, then $\chi(\alpha) = \chi'(\alpha)$ and so $\gamma$ fixes $\theta$, by \eqref{DpOnGaussSum}.  Then, by (i) and Lemma \ref{Stick}, we find that $\theta$ is a unit in $\cO_K$.  But $\vert \theta \vert_v = 1$ for all $v \vert \infty$ by \eqref{ConjChi}.  Hence $\theta$ is in the group of roots of unity $\Mu(\cO_K)$ of $\cO_K$.  

Let $p$ be odd, so that $\Mu(\cO_K) = \Mu_{q-1}$.  A Frobenius $\Phi =  \Frob_\gp$ acts on $\Mu_{q-1}$ by raising to the $p$-power and fixes each Gauss sum.  Thus $\theta^p = \Phi(\theta) = \theta$, so $\theta$ is in $\Mu_{p-1}$.  But $\theta \equiv \nu_q(c)/\nu_q(c') \equiv 1 \pmod{p}$ by Lemma \ref{Stick} and assumption (ii).  Since the elements of $\Mu_{p-1}$ are distinct mod $p$, we conclude that  $\theta = 1$.

Assume that $p = 2$, so $\Mu(\cO_K) =  \Mu_{2(q-1)}$.  Since Frobenius acts by squaring on $\Mu_{q-1}$ but fixes $\theta$, we find that $\theta$ is $\pm 1$.  By Remark \ref{Nu2} and assumption (ii), we have $\theta \equiv 1 \pmod{4}$ and so $\theta = 1$.
\end{proof}

\begin{Rem}
Suppose $p=2$.  If $f = 1$, we have only the trivial character {\bf 1}.  If $f = 2$, there are only two equivalence classes of characters modulo the action of Frobenius, namely {\bf 1} and the Teichm\"{u}ller character $\omega$, with $G(\omega) = 2$.  This covers the cases omitted in Proposition \ref{EqualG} when $p=2$.
\end{Rem}

\begin{Rem} \label{f=1}
For completeness, we recall that if $\chi$ and $\chi'$ are distinct characters on $\F_q^\times$ whose orders divide $p-1$, then $G(\chi) \ne G(\chi')$.  Clearly, $\chi$ and $\chi'$ are not primitive as soon as $f \ge 2$.   By \eqref{ConjChi}, we can assume that neither $\chi$ not $\chi'$ is trivial.  Define $n$ by $q-1 = (p-1)n$.  Since $\omega$ has order $q-1$ on $\F_q^\times$, we can write $\chi = \omega^{an}$ and $\chi' = \omega^{bn}$, where $1 \le a,b \le p-2$.  We have the $p$-adic digit expansion
$
an = a(1+p+\dots+ p^{f-1}), 
$
so $s_q(an) = af$ and similarly $s(bn) = bf$.  If $a$ and $b$ are not equal, Proposition \ref{EqualG} shows that $G(\chi) \ne G(\chi')$.
\end{Rem}

\begin{Ex}\label{Ex100} 
Using Proposition \ref{EqualG} for small $f$ and $p$ gives the following sample of fields $\F_q$ such that Gauss sum equalities only occur for Frobenius-conjugate characters.
\begin{enumerate}[i)]
\item Let $f = 3$.  If $p = 2, 3, 5, 7, 13, 17$, then all Gauss sum equalities are trivial.  If $p = 11$ and $\chi = \omega^{95}$ then $G(\chi) = G(\chi^{-1}) = (-11)^{3/2}$ by Theorem \ref{T:Gsum1}, but $\chi$ and $\chi^{-1}$ are not Frobenius-conjugate.  According to \cite{NZY}, it should be possible to distinguish these Gauss sums upon twisting the characters.  In fact, the digit sums already disagree for  $\chi \otimes \widehat{\lambda} = \omega^{228}$, where $\widehat{\lambda}$ is the lift of the twisting character $\lambda$ of order $p-1=10$.

  \vspace{2 pt}

\item Let $f = 5$.  If $p = 2, 3, 5, 7$, then all Gauss sums equalities are trivial. 

 \vspace{2 pt}

\item Even values of $f$ are precluded by Example \ref{Ex400}. 
\end{enumerate}
\end{Ex}

\vspace{5 pt}

For later use, we recall that certain values of the digit sums allow for better control over the field extension generated by the Gauss sum.  See \cite[Lemma 6]{Ev2} and \cite{Ol} for related information.  For odd primes $p$, let $\lambda\!: \, \F_p^\times \to \Mu_2$ be the quadratic character and let $p^* = \pm p \equiv 1 \pmod{4}$.  Since $G(\lambda)^2 = p^*$, define $\sqrt{p^*} = G(\lambda)$ in the quadratic subfield of $\Q(\Mu_{p})$.  

\begin{prop} \label{GoodpDigits}
Let $q = p^f$ for an odd prime $p$.  Let $\chi\!: \, \F_q^\times \to \Mu_M$, be a character of order $M$ and write $\chi = \omega^d$ in terms of the Teichm\"{u}ller character $\omega$.  If the $p$-adic digit sums $s_q(kd) = (p-1)f/2$ for all $k$ prime to $M$, then $G(\chi) = \eta \, \sqrt{p^*}^f$ with $\eta \in \Mu_{p-1} \cap \Mu_2 \, \Mu_M$.  Conversely, if there is a root of unity $\eta$ such that $G(\chi) = \eta \, \sqrt{p^*}^f$, then $s_q(kd) = (p-1)f/2$ for all $k$ prime to $M$.
\end{prop}

\begin{proof}
Assume that $s_q(kd) = (p-1)f/2$ for all $k$ prime to $M$.   According to Lemma \ref{Stick}, we have
$$
\ord_{\gp}(G(\chi)) = \frac{p-1}{2}f = f \ord_{\gp}(G(\lambda))
$$
 for all primes $\gp$ over $p$ in $L = \Q(\Mu_{pM})$.  Define $\eta$ by 
 $
 G(\chi) = \eta  \hspace{2 pt} G(\lambda)^f,
 $ 
 so that $\eta$ is a global unit.  By Lemma \ref{cs}(ii), $d \equiv s_q(d) \pmod{p-1}$ and so $d \equiv (p-1)f/2 \pmod{2}$, since $p-1$ is even.

By \eqref{DpOnGaussSum}, if a generator $\gamma$ for $\Gal(L/\Q(\Mu_M))$ satisfies $\gamma(\zeta_p) = \zeta_p^\alpha$, then   
$$
\gamma(G(\chi)) = \chi(\alpha^{-1}) \, G(\chi) = \omega(\alpha^{-d}) \, G(\chi) = (-1)^f G(\chi), 
$$
since $\alpha^{\frac{1}{2}(p-1)} = -1$ in $\F_p$.   Since $\gamma(\sqrt{p^*}) = - \sqrt{p^*}$, we find that $\gamma(\eta) = \eta$ and so $\eta$ is a unit in $\Z[\Mu_M]$.  It follows from \eqref{ConjChi} that $\vv{\eta}_v = 1$ at each archimedean place $v$.  Therefore $\eta$ is a root of unity in $\Z[\Mu_M]$ and so $\eta$ is in $\Mu_2 \, \Mu_M$.  Furthermore, the Frobenius $\Phi$ acts trivially on $G(\chi)$ and on $\sqrt{p^*}$, so also on $\eta$, giving $\eta = \Phi(\eta) = \eta^p$.  Hence $\eta \in \Mu_{p-1} \cap \Mu_2 \, \Mu_M$.  

The converse follows from Lemma \ref{Stick}.
\end{proof}

\section{Distinct Gauss sums} \label{DistinctionResults}
We use the notation of the introduction, with $q = p^f$ and $f \ge 2$.  Thus  $\chi$ and $\chi'$ are multiplicative characters $\F_q^\times \to \Mu_{q-1}$ and $\omega$ is the Teichm\"{u}ller character.   For $c$ in $\Z/(q-1)\Z$, it is convenient to write $G(c)$ instead of $G(\omega^c)$.  Since $G(c)$ is invariant under the action of Frobenius, we have $G(c') = G(c)$ whenever $c'$ is in the orbit $\lr{p}c = \{c, pc, p^2c, \dots\}$ of $c$ under multiplication by $p$ in $\Z/(q-1)\Z$.  As before, we say that an equality of Gauss sums $G(c) = G(c')$ is trivial if the characters $\omega^c$ and $\omega^{c'}$ are Frobenius-conjugate, i.e., if $c'$ is in $\lr{p}c$.   In this section, we obtain some simple consequences of a non-trivial equality of Gauss sums.

Suppose that $M$ divides $q-1$ and let $d = (q-1)/M$.  The distinct characters of order $M$ on $\F_q^\times$ comprise the set
$$
\cO_q(M) = \{ \omega^{u d} \mid u \in U_{q-1} \} =  \{\omega^{\ov{u} d} \mid \ov{u} \in U_M\},
$$
where $u \mapsto \ov{u}$ denotes the natural projection $U_{q-1} \twoheadrightarrow U_M$.  Thus, the group $H =  \Gal(\Q(\Mu_M)/\Q) \simeq U_M$ acts transitively on characters $\chi$ in $\cO_q(M)$ by $\sigma_u(\chi) = \chi^u$.  Let $[\chi]$ denote the Frobenius-conjugacy class of $\chi$ and let 
\begin{equation} \label{Obar}
\ov{\cO}_q(M) = \{\, [\chi] \, : \,  \chi \in \cO_q(M) \}.
\end{equation}
There is bijective correspondence $U_M/\lr{p} \leftrightarrow \ov{\cO}_q(M)$ by $a \lr{p} \leftrightarrow [\chi^a]$.   Let 
$$
\cG_q(M) = \{G(\chi) \mid \chi \in \cO_q(M)\}
$$ 
be the set of all Gauss sums for characters of order $M$.  There also is an induced action of $H/\lr{\Frob_\gp} \simeq U_M/\lr{p}$ on $\cG_q(M)$.  The following observation is an immediate consequence.

\begin{prop} \label{faithful}
All Gauss sum equalities between characters of order $M$ on $\F_q^\times$ are trivial if and only if the action of $U_M/\lr{p}$ on $\cG_q(M)$ is faithful. 
\end{prop}

The next Lemma easily follows from the transitive action of $H$ on $\cO_q(M)$ and will be used to fix one of the characters in an equality of Gauss sums.  

\begin{lemma} \label{CleanUp}
Suppose that $\chi \in \cO_q(M)$ and $d = (q-1)/M$.  If $G(\chi) = G(\chi')$, then there is some $u \in U_M$ such that $\sigma_u(\chi) = \omega^d$ and $G(\omega^d) = G(\sigma_u(\chi'))$.  Furthermore, $\chi$ and $\chi'$ are Frobenius-conjugate if and only if $\omega^d$ and $\sigma_u(\chi')$ are Frobenius-conjugate.  
\end{lemma}

\begin{prop} \label{gcd1}
Let $\chi$ and $\chi'$ be characters on $\F_q^\times$, with $\chi$ of order $q-1$.  If $G(\chi) = G(\chi')$, then $\chi$ and $\chi'$ are Frobenius-conjugate.
\end{prop}

\begin{proof}
By Lemma \ref{CleanUp}, we may assume that $\chi = \omega$.  If $\chi' = \omega^{d'}$, we must show that $d'$ is in $\lr{p}$.  By Proposition \ref{EqualG}, if $G(d') = G(1)$, then $s_q(d') = s_q(1) = 1$.  But $\lr{p}$ is the only orbit whose elements have a $p$-adic digit sum of 1.
\end{proof}

\begin{Ex}\label{Mersenne}
Let $\chi$ and $\chi'$ be character on $\F_q^\times$ with $\chi$ of {\em prime} order $q-1$, i.e., $p=2$ and $q-1$ is a Mersenne prime.  If $G(\chi) = G(\chi')$, then $\chi$ and $\chi'$ are Frobenius-conjugate.
\end{Ex}

For odd $p$, the mod $p$ congruence of Lemma \ref{Stick} on the unit part of the Gauss sum provides another method for distinguishing Gauss sums.   See \eqref{NuDef} for the function $\nu_q$.

\begin{prop} \label{gcdd}
Assume that $p > \min(f^d,d!) > 1$.  Let $\chi$ and $\chi'$ be characters on $\F_q^\times$, with $\chi$ of order $(q-1)/d$.  If $G(\chi) = G(\chi')$, then $\chi$ and $\chi'$ are Frobenius-conjugate.  
\end{prop}

\begin{proof}
By Lemma \ref{CleanUp}, we may assume that $\chi = \omega^d$.  Let $\chi' = \omega^{d'}\!,$ where $d'$ has the $p$-adic digit expansion $d' = a_0 + a_1 p + \dots + a_{f-1}p^{f-1}$.  The multinomial coefficient 
\begin{equation} \label{UnitRatio}
C = \binom{d}{a_0, \dots, a_{f-1}} = \frac{d!}{\nu_q(d')}  \vspace{2 pt}
\end{equation}
is bounded by $\min(f^d,d!) < p$.   By Proposition \ref{EqualG}, if $G(d') = G(d)$, we have $s_q(d') = s_q(d) = d$ and $\nu_q(d') \equiv \nu_q(d) \equiv d! \pmod{p}$, since $p \ge 3$.  Then $C \equiv 1 \pmod{p}$ and thus $C = 1$.   It follows that $d' = dp^j$ for some $j$, so $\chi = \omega^d$ and $\chi' = \omega^{d'}$ are Frobenius-conjugate.   
\end{proof}

\begin{cor} \label{gcd2}
If $G(\chi) = G(\chi')$ with $p$ odd and $\chi$ of order $(q-1)/2$, then $\chi$ and $\chi'$ are Frobenius-conjugate.  
\end{cor}

\begin{Ex} \label{3^n}
To extend Example \ref{Mersenne}, let $\chi$ and $\chi'$ be characters on $\F_q^\times$, where $q = p^f$ and $p = 3$ or 5, with the strong assumption that $\ell = \frac{q-1}{p-1}$ is prime.  For example, for $p = 3$, we have $(f,\ell) =  (5, 1093)$ or $(13,797161)$ and for $p = 5$, we have $(f,\ell) = (3,31)$ or $(7,19531)$.  Suppose that $G(\chi) =  G(\chi')$. Then, we can show that $\chi$ and $\chi'$ are Frobenius-conjugate.  Suppose not and consider the cases:

\vspace{2  pt}

$p=3$:   By Proposition \ref{gcd1} and Corollary $\ref{gcd2}$, neither character has order $q-1 = 2\ell$ nor order $(q-1)/2 = \ell$.  Hence both are quadratic characters and there is only one such.

\vspace{2  pt}

$p=5$:  Neither character has order $q-1 = 4 \ell$ nor $(q-1)/2 = 2\ell$, as above.  Also, if the order of both characters divides 4, we contradict Remark \ref{f=1}.  Thus, at least one of the characters, say $\chi$, has order $\ell$ and we can assume that $\chi = \omega^4$ by Lemma \ref{CleanUp}.  If $\chi' = \omega^{d'}$, then $s_q(d') = s_q(4) =  4$ by Proposition \ref{EqualG}.  Furthermore, the $5$-adic digit expansion of $d'$ gives a partition of $4$ such that 
$$
\nu_q(d') = \nu_q(4) = 4! \equiv  -1 \pmod{5}.
$$
We have
$$
\begin{array} { | c || c | c | c | c | c |}
\hline
\text{partition} & [4] & [3,1] & [2,2] & [2,1,1] & [1,1,1,1] \\
\hline
\nu_q & -1 & 1 & -1 & 2 & 1 \\
\hline
\end{array}
$$
Thus, the equality $G(\chi) = G(\chi')$ requires that $d' = 2+2\cdot 5^j$ for some $j$.  But then $s_q(\chi^2) = s_q(8) = 4$, while $s_q((\chi')^2) = s_q(4+4\cdot 5^j) = 8$, so not all requisite digit sums match.  Hence there are no non-trivial Gauss sum equalities for characters on $\F_q^\times$.    
\end{Ex}

\vspace{5 pt}

The question of distinguishing Gauss sums via {\em twisting} by a character $\eta$ on $\F_p^\times$ is of interest in representation theory.  See \cite{NZY} for conjectures and some general results, especially when $f < \frac{p-1}{2\sqrt{p}} +1$, obtained by cohomological methods.  The following application of our more elementary arguments gives a special case with no upper bound on $f$.  Note, however, that $f$ must be prime, so that this result verifies a special case of \cite[Conjecture 2.6]{NZY}.  

Recall that the twist of a character $\chi$ on $\F_q^\times$ by $\eta$ is $\chi \otimes \widehat{\eta} = \chi \, \widehat{\eta}$, where $\widehat{\eta}  = \eta \circ N$ is the lift of $\eta$ to $\F_q^\times$ via the norm $N\!: \, \F_q^{\times} \to \F_p^{\times}$.  Equivalently, $\widehat{\eta}$ has the form $\widehat{\eta} = \omega^{t (q-1)/(p-1)}$.

\begin{prop} \label{Twist}
Assume that $n  = \frac{q-1}{p-1}$ is prime and let $\chi$ and $\chi'$ be characters on $\F_q^\times$.  If $G(\chi \, \widehat{\eta}) = G(\chi' \,  \widehat{\eta})$ for all characters $\eta$ on $\F_p^\times$, then $\chi$ and $\chi'$ are Frobenius-conjugate.
\end{prop}

\begin{proof}
Suppose that $G(\chi \, \widehat{\eta}) = G(\chi' \,  \widehat{\eta})$ for all $\eta$.  If the order of both $\chi$ and $\chi'$ divides $p-1$, then $\chi = \chi'$ by Remark \ref{f=1}.  We may therefore assume that at least one of of these characters, say $\chi$, has the form $\chi = \omega^a$ with $n \nmid a$.  Also, $f \ge 2$, so $n \ge p+1$ and $n \nmid p-1$.    Take $\widehat{\eta} = \omega^{tn}$ with 
$$ 
t = \prod \{ \ell \text{ prime}, \, \ell \text{ divides } p-1 \text{ and } \ell \text{ does not divide } a \},
$$ 
allowing $t=1$ if the product is empty.  Let $\Psi = \chi \, \widehat{\eta} = \omega^c$ with $c = a + tn$ and observe that $\gcd(c,q-1) = 1$. Indeed, if $\ell$ is a prime dividing $q-1$, there are three cases. \vspace{2 pt}

\noindent \hspace{9 pt} i)  If $\ell = n$, then $\ell \nmid a$, so $\ell \nmid c$.  \hspace{7 pt} ii) If $\ell \vert p-1$ but $\ell  \nmid a$, then $\ell \vert t$, so $\ell \nmid c$.   \vspace{2 pt}

\noindent \hspace{3 pt} iii) If $\ell \vert p-1$ and $\ell \vert a$, then $\ell \nmid tn$, so $\ell \nmid c$. \vspace{2 pt}

\noindent Hence $\Psi$ has order $q-1$ and so, by Proposition \ref{gcd1}, $\chi' \, \widehat{\eta}$ is Frobenius-conjugate to $\chi \, \widehat{\eta}$.  But $\widehat{\eta}$ takes values in $\Mu_{p-1}$ so Frobenius at $p$ acts trivially on $\widehat{\eta}$.  It follows that $\chi$ and $\chi'$ are Frobenius-conjugate.
\end{proof}

Let $S_q(M) =  \{ s_q(ud) \, \vert \, u \in U_{q-1} \}$ be the set of $p$-adic digit sums that arise from the characters in $\cO_q(M)$.  If $\chi = \omega^c$, define $s_q(\chi) = s_q(c)$.  Our next results involve constraints on $\cG_q(M)$ and $S_q(M)$ when a non-trivial equality of Gauss sums exists.  If $\chi$ is in $\cO_q(M)$ and $\chi'$ is in $\cO_q(M')$, then the values of $\chi'$ are in $\Mu_{M'}$ and those of $\chi$ are in $\Mu_M$.  Thus, an equality of their Gauss sums implies that the common value $G(\chi) = G(\chi')$ lies in 
\begin{equation} \label{Intersect}
\Q(\Mu_{pM}) \cap \Q(\Mu_{pM'}) = \Q(\Mu_{pD}),
\end{equation}
where $D = \gcd(M,M')$.  Recall that $\phi$ denotes the Euler totient function.

\begin{prop} \label{S(M)}
Let $\chi$ and $\chi'$ be characters on $\F_q^\times$ of respective orders $M$ and $M'$.   Let $D = \gcd(M,M')$ and let $o_D(p)$ be the multiplicative order of $p$ in $U_D$.  If  $G(\chi) = G(\chi')$, then we have $\cG_q(M) = \cG_q(M')$, $S_q(M) = S_q(M')$ and $\vv{S_q(M)} \le \vv{\cG_q(M)}$.  In addition, $\vv{\cG_q(M)}$ divides $\phi(D)/o_D(p)$.
\end{prop}

\begin{proof}
The transitive action of $\Gal(\Q(\Mu_M)/\Q)$ on $\cO_q(M)$ produces all the characters of order $M$ from $\chi$ and thereby, all the corresponding $p$-adic digit sums in $S_q(M)$.  By \eqref{Intersect}, if $G(\chi) = G(\chi')$, there is an induced transitive action of $H = \Gal(\Q(\Mu_D)/\Q)$ on both $\cG_q(\chi)$ and $\cG_q(\chi')$.  We have $\cG_q(\chi) = \cG_q(\chi')$ and likewise, $S_q(M) = S_q(M')$.  Since $\Frob_\gp$ fixes each Gauss sum, the action of $H$ on $\cO_q(M)$ factors through $\ov{H} = H/(\lr{\Frob_\gp} \cap H) \simeq U_D/\lr{p}$.  If $\ov{H}_\chi$ is the subgroup of $\ov{H}$ that stabilizes $G(\chi)$, then 
$$
\vv{\cG_q(M)} = \fdeg{\ov{H}}{\ov{H}_\chi}  \quad \text{divides} \quad \vv{\ov{H}} = \phi(D)/o_D(p).
$$
The same $p$-adic digit sum may occur for characters with different Gauss sums, so $\vv{S_q(M)} \le \vv{\cG_q(M)}$.
\end{proof}

\begin{prop} \label{gcdProp}
Let $\chi$ and $\chi'$ be characters on $\F_q^\times$ of respective orders $M$ and $M'$ and assume that $\gcd(M,M') = 1$.  If  $G(\chi) = G(\chi')$, then $f = 2n$ is even, $G_q(\chi) = G_q(\chi') = \pm \,  p^n$ and $s_q(\chi) = s_q(\chi') =  (p-1)n$.
\end{prop}

\begin{proof}
By \eqref{Intersect}, $G(\chi) = G(\chi')$ lies in $F = \Q(\Mu_p)$.  Recall from \S \ref{Criterion} that $\Gal(F/\Q) =  \lr{\gamma}$, with $\gamma(\zeta_p) = \zeta_p^\alpha$, where $\zeta_p$ is a primitive $p$-th root of unity and $\alpha$ is a primitive root modulo $p$.  The equality of Gauss sums implies that $\chi(\alpha^{-1}) = \chi'(\alpha^{-1})$ by \eqref{DpOnGaussSum}.  But then $\chi(\alpha^{-1})$ lies in $\Mu_M \cap \Mu_{M'} = \{1\}$.  Hence $G(\chi)$ is fixed by $\gamma$ and so $G(\chi) \in \Q$.  Since $\vv{G(\chi)}_\infty = \sqrt{q}$ at the archimedean place, we conclude that $f = 2n$ is even and $G(\chi) = \pm p^n$.  The unique digit sum $(p-1)n$ follows from Lemma \ref{Stick}
\end{proof}

\begin{Ex}
Suppose the $p^n + 1 = c_1 M_1 = c_2 M_2$ and let $q = p^{2n}$.  For $j = 1,2$, the character $\chi_j = \omega^{(q-1)/M_j}$ has order $M_j$.  If we arrange for $c_1$ and $c_2$ to be even, then an evaluation of Gauss sums due to Stickelberger gives $G(\chi_1) = G(\chi_2) = p^n$ as in Proposition \ref{StickelEval}.  Frobenius-conjugate characters have the same order, since $p^j$ is relatively prime to $q = p^f-1$.  But it is easy to find examples where $M_1 \ne M_2$, in which case $\chi_1$ and $\chi_2$ cannot be Frobenius-conjugate.   Thus, we have a non-trivial Gauss sum equality.  Here are some small numerical examples with $n = 3$:
$$
\begin{array}{| c || c | c | c | c |}
\hline
p & 5  & 11 & 17  & 17 \\
\hline
p^3+1 &  2 \cdot 7 \cdot 9 &  4 \cdot 9 \cdot 37 &  2 \cdot 7 \cdot 13 \cdot 27 &  2 \cdot 7 \cdot 13 \cdot 27 \\
\hline
M_1 &  7 & 9 & 7 & 27  \\
\hline
M_2 &   9 & 37 & 13 & 91  \\
\hline
\end{array}
$$
It is likely that in these examples, the Gauss sums become unequal when we twist by the lift of a character of order $p-1$.
\end{Ex} 
 
 \vspace{5 pt}
 
\begin{Ex} \label{Ex200}
The case $q = 3^5$  was treated in Example \ref{Ex100}, by straightforward computation of the criteria in Proposition \ref{EqualG}.  However, we now verify that the only Gauss sum equalities are the trivial ones by using some of the refinements achieved in this section.   Let $M$ be a divisor of $q-1 =  2 \cdot 11^2$.  Recall from \eqref{Obar} that $\ov{\cO}_q(M)$ denotes the set of equivalence classes of characters modulo Frobenius-conjugacy and that $U_M/\lr{p}$ acts transitively on $\ov{\cO}_q(M)$.  Suppose that $G(\chi) = G(\chi')$ is a non-trivial equality, with $\chi$ and $\chi'$ of order $M$ and $M'$ respectively.  We rule out orders $q-1$ and $(q-1)/2$ by Proposition \ref{gcd1} and Corollary \ref{gcd2}. Hence $M,M' \in \{2,11,22\}$.  If $M  =  11$ or $22$, the group $U_{M}/\lr{3}$ is generated by the coset of $-1$.  The following table shows that the digit sums are sufficient to distinguish among the Frobenius-conjugacy classes of characters.   \vspace{5 pt}
$$
\quad  \begin{array}{ | c || r  || r r || r r | }
\hline
&\ov{\cO}_q(2):   \hspace{3 pt}  [\omega^{121}]  &  \ov{\cO}_q(11):    \hspace{3 pt}  [\omega^{22}] & [\omega^{-22}] &  \ov{\cO}_q(22):   \hspace{3 pt}  [\omega^{11}] & [\omega^{-11}] \\
\hline
s_q  & 5 \quad & 4 \quad & 6 \quad &  3 \quad & 7 \quad \\
\hline
 \end{array}
 $$
Thus, there are only trivial Gauss sum equalities for $q = 3^5$.
\end{Ex}

\vspace{5 pt}

For $\chi$ in $\cO_q(M)$, write $\chi = \omega^c$ in terms of the Teichm\"{u}ller character $\omega$ on $\F_q^\times$ and recall that by definition, $s_q(\chi) = s_q(c)$. Furthermore, the Gauss sums $G(\chi) = G(\chi')$ and the $p$-adic digit sums $s_q(\chi) = s_q(\chi')$ are equal when $\chi$ and $\chi'$ are Frobenius-conjugate.   As a variant of Proposition \ref{S(M)}, we can keep track of the number of characters modulo Frobenius-conjugacy with the same $p$-adic digit sum.   For each element $s$ in $\{ 1, 2, ..., f(p-1) \}$, define the {\em multiplicity}:
$$
\mult(M,s) = \#\, \{ \, [\chi] \in \ov{\cO}_q(M) \, : \,  s_q(\chi) = s  \, \}.
$$ 
In particular, $\mult(M,s) = \mult(M,(p-1)f-s)$ by Lemma \ref{cs}(ii) 

\begin{prop}\label{sameorderclass}
Assume that 
$
\gcd \{\mult(M,s) : 1 \le s \le f(p-1) \} = 1.
$  
If $\chi$ and $\chi'$ are characters of order $M$ on $\F_q^\times$ with equal Gauss sums, then $\chi$ and $\chi'$ are Frobenius-conjugate.  
\end{prop}

\begin{proof}
As in Proposition \ref{S(M)}, the group $\ov{H} = \Gal(\Q(\Mu_M)/\Q)/\lr{\Frob_\gp}$ acts transitively on $\ov{\cO}_q(M)$ and so the multiplicity of Frobenius-conjugacy classes $[\Psi]$ with a fixed Gauss sum is the order of the stabilizer $\ov{H}_\Psi$ of $G(\Psi)$.  This is independent of the choice of $\Psi$ in $\cO_q(M)$.  Since different Gauss sums may have the same $p$-adic digit sum, $\mult(M,s)$ is a multiple of $\vv{\ov{H}_\Psi}$.  But if $\gcd\{\mult(M,s) \} = 1$, then $\ov{H}_\Psi$ is trivial.  Hence equalities of Gauss sums occur for characters of order $M$ only if the characters are Frobenius-conjugate.
\end{proof}

\begin{Ex} \label{Ex300}
Let $q = 11^4$ and consider characters of order $M = 305$ on $\F_q^\times$.   Since $\ov{H} \simeq U_M/\lr{11}$ is cyclic of order 60, generated by the coset of $2$, there are 60 Frobenius-conjugacy classes comprising $\ov{\cO}_q(M)$ and one can compute the multiplicities of the corresponding digit sums $s_q(48 \cdot 2^j)$.  Finding that $\mult(M,4) = 1$ suffices, but we give the full table for completeness.
$$
\begin{array}{ | c || c | c | c | c | c | c | c | c | c | c | c | c}
\hline
s  = s_q & 4, \, 36 & 8,\, 32 & 12, \, 28 & 14, \, 26 & 16, \,  24 &18, \, 22 \\ 
\hline
\mult(M,s) & 1 & 3 & 6 & 4 & 10 & 6 \\
\hline
\end{array}
$$
Therefore, there are no non-trivial Gauss sum equalities between characters of order $M =  305$.

\end{Ex}

\numberwithin{equation}{section}
\section{Hasse-Davenport, $p$-adic Gamma and Gross-Koblitz} \label{HDGK}
In this section, we recall two important formulas for the evaluation of Gauss sums.  The Hasse-Davenport formula \cite[p.~163]{IR} applies when a character on $\F_q$ is lifted to a larger finite field.   

The Gross-Koblitz formula \cite{GK} provides an evaluation of Gauss sums in terms of $p$-adic Gamma functions.  See \cite{C} for a complete proof, including the case $p = 2$.   It is used to obtain the 2-adic congruences required in Remark \ref{Nu2} and in the evaluation of Gauss sums in Theorem \ref{T:Gsum1}.

\vspace{5 pt}

\noindent{\bf Hasse-Davenport lifting}.
Consider the field extension $\F_{q^n}/\F_q$ of degree $n$. If $\chi$ is a character of $\F_q^\times$, let $\chi^{(n)}$ be the character of $\F_{q^n}^\times$ given by $\chi^{(n)}=\chi \circ N$, where $N\!: \, \F_{q^n} \to \F_q$ is the norm.  Consistent with our definition of the Gauss sums, we use the additive characters 
$$
\psi(a) = \zeta_p^{\Tr_{\F_q/\F_p}(a)} \text{ on } \F_q \, \text{ and } \, \psi^{(n)}(a) = \psi\circ \Tr_{\F_{q^n}/\F_q}(a) =  \zeta_p^{\Tr_{\F_{q^n}/\F_p}(a)}  \text{ on } \F_{q^n}.
$$
The Hasse-Davenport lifting relation asserts that 
\begin{equation} \label{HD}
G(\chi^{(n)}) = (-1)^{n-1}G(\chi)^n.
\end{equation}

\vspace{5 pt}

\noindent {\bf The $p$-adic Gamma-function}.  
Let $\vert \hspace{4 pt} \vert_p$ be the standard $p$-adic norm on $\Q_p$, satisfying $\vert p \vert_p = \frac{1}{p}$ and let $\Z_p^{\times}$ denote the group of units of $\Z_p$.  The $p$-adic Gamma-function $\Gamma_p: \Z_p\rightarrow \Z_p^\times$ is defined by
$$
\Gamma_p(x) = \lim_{m \in \N, \, m \to x}(-1)^m \, \prod \{ k \,\, \vert \,\, 0 < k < m, \, (k,p)=1\}.
$$
In particular, $\Gamma_p(0) = 1$ and  $\Gamma_p(1) = -1$.  A Lipschitz condition \cite[p.~370]{Coh}  holds if $p$ is odd, or if $p = 2$ and $\vv{x - y}_2 \leq 1/8$:
\begin{equation} \label{Lip}
\vv{\Gamma_p(x)-\Gamma_p(y)}_p \le \vv{x-y}_p.
\end{equation}
There is a functional equation: 
\begin{equation} \label{FuncEq}
\Gamma_p(x+1)/\Gamma_p(x) = \begin{cases} -x &\text{if } \vv{x}_p = 1, \\ -1 &\text{if } \vv{x}_p < 1. \end{cases}
\end{equation}

For odd $p$ and $x$ in $\Z_p$, let $L(x)$ be the unique representative in $\{1, \dots, p\}$ for $x \bmod{p}$ and define $L_1(x) = (x-L(x))/p$.    If $p = 2$, write $x \equiv x_0 + 2x_1 \pmod{4}$, with $x_0, \, x_1 \in \{0,1\}$.  The {\em reflection principle} for the $p$-adic Gamma-function depends on the parity of $p$, as follows:
\begin{equation} \label{reflec}
\Gamma_p(x) \, \Gamma_p(1-x) = \begin{cases} (-1)^{L(x)} &\text{if } p \text{ is odd}, \\ 
(-1)^{1+x_1}   &\text{if } p = 2. \end{cases}
\end{equation}

Suppose that $m$ is a unit in $\Z_p$ and assume that the integer $t$ is prime to $p$.  Then $m^{p-1} \equiv 1 \pmod{p}$ and so $\left(m^{p-1}\right)^{1/t}$ is defined by the binomial expansion around 1.  In particular, if $p$ is odd and $t$ divides $p-1$ then, in terms of the Teichm\"{u}ller character $\omega$, we have
\begin{equation*} 
\frac{m^{\frac{p-1}{t}}}{\left(m^{p-1}\right)^{1/t}} = \omega(m^{\frac{p-1}{t}}) \, \in \, \Mu_t.
\end{equation*} 

Using this notation with $m = 2$, the {\em duplication formula} for odd $p$ is
\begin{equation} \label{Dup}
\Gamma_p(x) \, \Gamma_p(x + {\textstyle \frac{1}{2}}) = 2^{1-L(2x)} \, \left(2^{1-p}\right)^{L_1(2x)} \, \Gamma_p\left({\textstyle \frac{1}{2}}  \right) \, \Gamma_p(2x).
\end{equation}

\begin{lemma}  \label{GammaMod4}
Let $c = c_0 + 2c_1 + 4c_2 \dots$ be the $2$-adic expansion of $c$.  The value of $\Gamma_2(c) \bmod{4}$ only depends on $c \bmod{8}$, as follows:
$$
\begin{array} {| c || r | r | r |  r | r | r |  r | r | }
\hline
c \bmod{8}  & 0 & 1  & 2 & 3 & 4 & 5 & 6 & 7  \\
\hline
\Gamma_2(c) \bmod{4} & 1 & 3 & 1 & 3 & 3 & 1 & 3 & 1 \\
\hline
\end{array}
$$
Moreover, $\Gamma_2(c) \equiv (-1)^{c_0+c_2} \pmod{4}$ and $\Gamma_2(-c) \equiv (-1)^e \pmod{4}$, where $e = c_1^2+c_2^2+c_0 \, c_1$.
\end{lemma}

\begin{proof}
According to \eqref{Lip}, to determine $\Gamma_2(c) \pmod{4}$, it suffices to let $c$ run over a complete residue system mod 8, as in table above.  Then the Lemma follows by direct computaton.
\end{proof}

\vspace{5 pt}

\noindent {\bf The Gross-Koblitz formula} 
Recall that $\{x\}$ is the fractional part of $x$, chosen to satisfy $0 \le \{x\} < 1$.  We preserve some of the notation from the start of \S\ref{Criterion}:
\begin{enumerate} [\hspace{2 pt} $\bullet$]
\item     $q = p^f$, $L = \Q(\Mu_{q-1}, \Mu_p)$ and $\gp$ is a prime over $p$ in the ring of integers $\cO_L$ with residue field $k_L$. \vspace{2 pt}
\item     $\omega\!: \, \F_q^\times \xrightarrow{\sim} k_L^\times \simeq \Mu_{q-1}$ is the Teichm\"uller character, such that $\omega(a)$ is the unique root of unity in $\Mu_{q-1}$ congruent to $a$ modulo $\gp$.  
\end{enumerate}
Our definition of the Gauss sum $G(-c) = G(\omega^{-c})$ depends on a fixed primitive $p$-th root of unity $\zeta_p$.   Let $\C_p$ denote the completion of a fixed algebraic closure of $\Q_p$.  There is a unique element $\pi$ in $\C_p$ satisfying $\pi^{p-1} = -p$ and $\zeta_p \equiv 1+\pi \pmod{\pi^2}$.  The formula of Gross and Koblitz asserts that
\begin{equation} \label{Gross-Koblitz}
G(-c) = -\pi^{s_q(c)}\prod_{i=0}^{f-1}\Gamma_p\left(\left\{\frac{cp^i}{q-1}\right\}\right).
\end{equation}

Let $a = a_0 + a_1 p + \dots + a_{f-1} p^{f-1}$ be the $p$-adic digit expansion, with the usual conventions of Notation \ref{DigitSum}.  For $0 \le j \le f-1$, let 
$$
a^{(j)} = a_j + pa_{j+1} + \dots + p^{f-1} a_{f-1+j}
$$ 
denote the permuted expansion, treating the subscripts modulo $f$.  Since $p^j a \equiv a^{(j)} \bmod{(q-1})$, with $0 \le a^{(j)} < q-1$, the fractional part in the Gross-Koblitz formula evaluates to
\begin{equation} \label{GammaIdent}
\left\{\frac{p^j a}{q-1}\right\} = \frac{a^{(j)}}{q-1} \quad \text{and so} \quad G(-a) = -\pi^{s_q(a)}\prod_{i=0}^{f-1}\Gamma_p\left(\frac{a^{(j)}}{q-1}\right).
\end{equation}

\vspace{5 pt}

Next, we examine the product that appears in the Gross-Koblitz formula when $p=2$, namely
\begin{equation} \label{sn}
\eta(a,f) = \prod_{j=0}^{f-1} \, \Gamma_2\left(\left\{\frac{2^j a}{q-1}\right\}\right) = \prod_{j=0}^{f-1} \, \Gamma_2\left(\frac{a^{(j)}}{q-1}\right).
\end{equation}
To express the value of $\eta(a,f)$  modulo 4, define $\beta_q(a) = \sum_{j = 0}^{f-1} \, a_j \, a_{j+1}$ with $a_f = a_0$, as in Notation \ref{DigitSum}.

\begin{lemma}  \label{2AdicCong}
If $f \ge 3$, then $\eta(a,f) \equiv (-1)^{\beta_q(a)} \bmod{4}$.
\end{lemma}

\proof
If $f \ge 3$, then $q-1 \equiv -1 \bmod{8}$, so $\eta(a,f) \equiv \prod_{j=0}^{f-1} \Gamma_2(-a^{(j)}) \bmod{8}$.  Then, by Lemma \ref{GammaMod4}, we find that $\eta(a,f) \equiv (-1)^T \bmod{4}$, where
$$
T \equiv \sum_{j = 0}^{f-1} (a_j a_{j+1} + a_{j+1}^2 + a_{j+2}^2) \equiv \sum_{j = 0}^{f-1} a_j a_{j+1}  \equiv \beta_q(a) \pmod{2}.       \hspace{20 pt}   \qed
$$

\numberwithin{equation}{section}
\section{Examples with explicit Gauss sums} \label{Explicit}

Throughout this section, $p$ is an odd prime and $q = p^f$.   Recall that $p^* = \pm p$ with $p^* \equiv 1 \pmod{4}$.   If $\lambda\!: \, \F_p^\times \to \Mu_2$ is the quadratic character, then $G(\lambda)^2 = p^*$, so $G(\lambda)$ is a convenient choice for $\sqrt{p^*}$ in the quadratic subfield of $\Q(\Mu_p)$.  R. J. Evans evaluated certain {\em pure} Gauss sums, i.e., those for which a non-trivial integer power is real.  As usual, write $U_m$ for the group of units in the ring $\Z/m\Z$.  Let $\lr{a}$ be the cyclic subgroup of $U_m$ generated by $a$ and let $o_m(a) = \vv{\lr{a}}$ be the order of $a$ in $U_m$.  

\begin{theorem}\cite[Cor. 8]{Ev}.   \label{T:Gsum1}
Let $\chi\!: \, \F_q^\times \to \Mu_{2m}$ be a character of order $2m$, with $m$ odd.  If $2$ is in the subgroup $\lr{p}$ of $U_m$, then 
\[
	G(\chi) = (-1)^{f-1} G(\lambda)^f  = (-1)^{f-1}(p^{*})^{f/2}.
\]
In particular, $G(\chi^n) = G(\chi)$ for all odd $n$.
\end{theorem}

\begin{Rem}
If $\widehat{\lambda}  =  \lambda \circ N_{\F_q/\F_p}$ is the lift of $\lambda$ to $\F_q^\times$, then  $G(\chi) =  G(\widehat{\lambda})$ by Hasse-Davenport \eqref{HD}.
\end{Rem}

\begin{Rem}
The choice of $m$ determines the values of $f$ and the primes $p$ to which the Theorem applies.  Indeed, let $\lr{a}$ be any  cyclic subgroup of $U_m$ containing 2 and choose any odd prime $p \equiv a \pmod{m}$.  Then take $f$ to be any multiple of $o_m(a)$. 
\end{Rem}

\vspace{5 pt}

Using Theorem \ref{T:Gsum1} and the conventions of Gauss's evaluation for quadratic characters \cite[p.~362]{BEW}, the complex embedding of $G(\chi)$ is easily determined, as follows:

\begin{cor}
$
G(\chi) = \begin{cases}
\hspace{4 pt} (-1)^{f-1} \, q^{1/2} &\text{if }p\equiv 1 \pmod{4}, \vspace{2 pt} \\
(-1)^{f-1} \, i^{f} \,q^{1/2}  &\text{if }  p\equiv 3 \pmod{4}.
\end{cases}
$
\end{cor}

\vspace{5 pt}

We also quote the evaluation of certain Gauss sums by Stickelberger \cite[p.~364]{BEW},  allowing $p =2$ and any $M \ge 3$. 

\begin{prop}[Stickelberger] \label{StickelEval}
Suppose that $p$ has even order $f = 2n$ in $U_M$ and $p^n \equiv -1 \pmod{M}$.  Let $q=p^{f}$, let $\chi$ be a character of order $M$ on $\F_{q}^\times$ and write $p^n+1 = cM$.  Then
$$
G(\chi) = G(\chi^{-1}) = \begin{cases} (-1)^c \, p^n &\text{if } p \text{ is odd,} \\
                                                              2^n &\text{if } p = 2 \text{ and } n > 2.
                                        \end{cases}
$$ 
\end{prop}

Theorem \ref{T:Gsum1}  and Proposition \ref{StickelEval} give non-trivial Gauss sum equalities such that the digital sum is $f(p-1)/2$ and $G(\chi)$ is in $\mathbb{Q}$ or a quadratic field. 

\begin{Ex} \label{Ex400}
Assume that $f$ is even, $q = p^f$ and $d = p^{f/2}-1$.  Then $\chi = \omega^d$ is a character of order $M = p^{f/2}+1$ on $\F_q^\times$ and $G(\chi) = G(\chi^u)$ is rational for all $u \in U_M$ by Proposition \ref{StickelEval}.  As soon as there is a choice of $u \notin \lr{p}$, namely if $q  \notin \{2^2, 3^2\}$, the characters $\chi$ and $\chi^u$ are not Frobenius-conjugate.  
\end{Ex}

However, neither condition in Example \ref{Ex400} is necessary.

\begin{Ex} \label{Ex500}
Let $\chi  = \omega^d$.  The generator $\gamma$ of $\Gal(\Q(\Mu_{p})/\Q)$ in \eqref{DpOnGaussSum} acts on $G(\chi)$ by $\gamma(G(\chi)) = \chi(\alpha^{-1}) \, G(\chi)$, where $\alpha$ is a primitive root modulo $p$, i.e. the multiplicative order $o_p(\alpha)  = p-1$.  Hence the ramification index of $p$ in $\Q(G(\chi))/\Q$ is given by $e = (p-1)/(d,p-1)$ and so an extension of degree at least $e$ is required to obtain $G(\chi)$.  In the following examples, Proposition \ref{EqualG} can be used to verify that if $\chi' = \omega^{d'}$, then $G(\chi) = G(\chi')$.  We have $s_q(d) = s_q(d') \ne f(p-1)/2$ and $e > 2$.
$$
\begin{array}{| c | c | c | c | c | c | c |}
\hline
p & f & d &  d' & s_q(d) = s_q(d') & \nu_q(d) = \nu_q(d')  &   e \\
\hline
19 & 2 & 12 & 156 & 12 & 17 & 3 \\ 
\hline
41 & 2 & 28 & 308 &  28 & 35 &10 \\
\hline
\end{array}
$$
\end{Ex}

\vfill

\end{document}